\documentclass[12pt]{amsart}

\usepackage{amsmath,amsthm,amsfonts,latexsym,amssymb,enumerate,color}
\usepackage{graphicx}
\usepackage{enumitem}

\newcommand\CC{{\mathbb C}}

\newcommand\DD{{\mathbb D}}

\newcommand\GN{{\mathcal N}}
\newcommand\NN{{\mathbb N}}

\newcommand\RR{{\mathbb R}}
\newcommand\QQ{{\mathbb Q}}
\newcommand\GS{{\mathcal S}}
\newcommand\MS{\mathbb S}
\newcommand\TT{{\mathbb T}}

\def\beq{\begin{equation}}
\def\eeq{\end{equation}}
\newtheorem{thm}{Theorem}[section]

\newtheorem{lem}[thm]{Lemma}
\newtheorem{cor}[thm]{Corollary}
\newtheorem{rem}[thm]{Remark}

\newtheorem{ex}[thm]{Example}
\newcommand\beginpf{\noindent {\bf Proof:} \quad}

\newcommand\re{\mathop{\rm Re}\nolimits}
\newcommand\im{\mathop{\rm Im}\nolimits}

\def\beginpf{\begin{proof}}
\def\endpf{\end{proof}}

\renewcommand\phi{\varphi}

\newcommand\ol{\overline}
\newcommand\Hol{\mathop{\rm Hol}}

\newcommand\Poa{\Psi_{\omega,\alpha}}

\begin{document}
\title[Phase retrieval on circles and lines]{Phase retrieval on circles and lines}

 \author{I. Chalendar}
\address{Isabelle CHALENDAR, Université Gustave Eiffel, LAMA, (UMR 8050), 
    UPEM, UPEC, CNRS, F-77454, Marne-la-Vallée (France)}
\email{isabelle.chalendar@univ-eiffel.fr}

\author{J. R. Partington}
\address{Jonathan R. PARTINGTON, School of Mathematics, University of Leeds, Leeds LS2 9JT, Yorkshire, U.K.}
\email{j.r.partington@leeds.ac.uk}


\baselineskip18pt

\bibliographystyle{plain} 

\begin{abstract}
Let $f$ and $g$  be analytic functions  on the open unit disc $\DD$ such that $|f|=|g|$ on a set $A$.  We give an alternative proof of the result of Perez that there exists  $c$ in the unit circle $\TT$ such that  $f=cg$ when $A$ is the union of two lines in $\DD$ intersecting at an angle that is an irrational multiple of $\pi$, and from this deduce a sequential generalization of the result.
Similarly,
 the same conclusion is valid when $f$ and $g$ are in the Nevanlinna class and 
$A$ is the union of the unit circle and an interior circle, tangential or not.
We also provide sequential versions of this result and analyse the case $A=r\TT$. 
Finally, we examine the most general situation when there is equality on two distinct circles in the disc, proving a result
or counterexample for each possible configuration.
 \end{abstract}

 \maketitle
 
 \section{Background}
 
 A fundamental question of phase retrieval is the following:
 {\em For which measurable sets $A \subset \overline\DD$ is an analytic function
 $f \in H^1$ determined uniquely to within a unimodular constant by the
 values of $|f|$ on $A$?} Here $\DD$ denotes the unit disc and $H^1=H^1(\DD)$ the 
 associated Hardy
 class of analytic functions.
 Such questions arise in signal processing, crystallography, quantum mechanics, microscopy, and
 many other applications.\\

 In \cite{PLB17} Pohl,  Li and  Boche worked with the class $H^1_R$ of functions
 in $H^1=H^1(\DD)$ such that the inner factor is a Blaschke product (no singular part).
 They showed (their Theorem 3) that any such function $f$ is determined to within a
 unimodular constant by the values of $|f|$ on $\TT$ and $r\TT$ for any $0<r<1$.
 This restriction on the inner factor was removed by Perez \cite{perez}.

 Then in \cite{JKP20}
Jaming, Kellay and  Perez solved a phase retrieval problem in $L^2(\RR, e^{2c|x|} \, dx)$
for $c>0$. This is isometric by the Fourier transform to the Hardy space of a strip 
$\GS = \{z \in \CC: | \im z| < c \}$ (similar calculations occur in \cite{smith}).

By conformal mapping the phase retrieval problem can then be reformulated on $H^2(\DD)$
in terms of finding an expression for all pairs $F,G \in H^2$ such that $|F|=|G|$ on $(-1,1)$.
For example, with the notation $F^*$ for the function defined by
$F^*(z)=\overline{F(\bar z)}$, they showed that
for $F,G \in H^2$ one has $|F|=|G|$ on $(-1,1)$ if and only if there exist
$u,v \in \Hol(\DD)$ (the space of all
holomorphic functions on $\DD$) such that $F=uv$ and $G=uv^*$.

Earlier results along similar lines are due to McDonald \cite{M04}, who
studied the problem of determining an entire function $f$ when $|f (x)|$
is known for every real $x$.

Finally, a  recent paper of Liehr \cite{liehr} looks at Gabor phase retrieval and Pauli-type uniqueness problems.\\

This paper deals with several cases, for which it provides positive and negative results: intersecting curves, disjoint circles, disjoint lines, and
a full analysis of the two-circle situation. We also show that many results have sequential counterparts, thus extending and generalizing the work of \cite{perez}.
 
 \section{Results}
 
 \subsection{Intersecting curves}\label{sec:2.1}
 
Perez \cite{perez} proved a version of the following result, which generalizes  \cite[Lem.\ 4.5]{JKP20}.
We give an independent  proof, introducing a method that
allows us to provide further extensions.

\begin{thm}\label{thm:first}
Let $f, g: \DD \to \CC$ be  analytic functions and $L_1$, $L_2$ two line segments
in $\DD$,
intersecting at an angle $\theta \in (0, \pi/2)\setminus \pi \QQ$,
such that $|f(z)|=|g(z)|$ for all $z \in L_1 \cup L_2$.
Then $f=\beta g$ for some constant $\beta \in \TT$.
\end{thm}

\beginpf
Clearly by considering the functions on a small disc $\{z \in \CC: |z-a| \le r \}$ where 
$a$ is the intersection point of $L_1$ and $L_2$, and by translating and rescaling
as necessary, we may suppose that $f$ and $g$ lie in the disc algebra and the two
lines intersect at $0$.

We may also suppose that the functions are non-vanishing except possibly at $0$.

If they do vanish at $0$, then the orders of the zeros are equal, since for some $N$
$\lim_{z \to 0} |f(z)|/|z|^N $ exists and is nonzero, and the same for $g$. So by
dividing out any the zeros at $0$, and restricting to a smaller disc if necessary, we may
suppose that $f$ and $g$ are invertible functions in the disc algebra.

Now $h:=f/g$ is also in the disc algebra and it has modulus $1$ on both lines $L_1$ and $L_2$.
Thus $h$ maps $L_1$ and $L_2$ into the unit circle in the complex plane.

Suppose that for some $k \ge 1$ we have
\[
h(z)=c_0+c_k z^k + O(z^{k+1})
\]
near $0$, with $|c_0|=1$ and $c_k \ne 0$; 
then near to $0$ it magnifies angles by a factor of $k$, and so
if $h$ is non-constant, then
$k \theta$ must be an integer multiple of $\pi$.
\endpf

\begin{ex}\label{ex:22}
Take $L_1=[-1,1]$ and $L_2=i[-1,1]$. Then 
\[
f(z)=\frac{z^2-2i}{z^2+2i} \qquad \hbox{and} \qquad g(z) =\frac{z^2-3i}{z^2+3i}
\]
are both unimodular on $L_1 \cup L_2$. Similar examples can be constructed for other rational multiples of $\pi$.
\end{ex}

\begin{cor}\label{cor:c1curves}
The conclusions of Theorem \ref{thm:first} also hold if $L_1$ and $L_2$ are two $C^1$ curves
intersecting at an angle $\theta \in (0, \pi/2]\setminus \pi \QQ$. 
\end{cor}
\beginpf
The conformality argument works with minor changes.
\endpf

We may prove sequential versions of the above results with the
aid of the following easy lemma.

\begin{lem}\label{lem:little}
Let $h$ be a function holomorphic  on a neighbourhood of $0$, 
$\theta \in [0,2\pi]$ and $(a_n)$ a real sequence
such that $a_n \to 0$ and  $h(a_n e^{i\theta})$ is real for all $n$. Then
all the derivatives of the function $z \mapsto h(z e^{-i\theta})$ at $0$ are real.
\end{lem}
\beginpf
Clearly we may assume without loss of generality
that $\theta=0$. Then
$h'(0) = \lim_{n \to \infty}  h(a_n)/a_n  \in \RR$.
    The map $h_1 : z \mapsto h(z)/z - h'(0)$ is also holomorphic around $0$, and satisfies $h_1(a_n) \in \RR$. Then, in a same way, we get $h_1'(0) = h''(0) \in \RR$.
     Iterating this reasoning, we see that all the derivatives of $h$ at $0$ must be real.
     \endpf
     
     \begin{cor}
Let $f, g: \DD \to \CC$ be  analytic functions and $L_1$, $L_2$ two line segments
in $\DD$,
intersecting at a point $p$, at an angle $\theta \in (0, \pi/2)\setminus \pi \QQ$,
such that there are sequences $(a_n) \subset L_1$ and $(b_n) \subset L_2$
tending to $p$
such that $|f(a_n)|=|g(a_n)|$ and $|f(b_n)|=|g(b_n)|$ for all $n$.
Then $f=\beta g$ for some constant $\beta \in \TT$.
\end{cor}
\beginpf
We use the argument of the proof of Theorem \ref{thm:first} up to the point
where we have that $h(a_n)$ and $h(b_n)$ lie in $\TT$ for all $n$.
By composing with   suitable bilinear transformations we can construct a function $\tilde h$
holomorphic near $0$ with 
sequences $c_n \to 0$ and $d_n \to 0$ such that
the $c_n$ are real and the $d_n$ have argument $\theta$, but with
$\tilde h(c_n)$ and $\tilde h(d_n)$ both real for each $n$.
By Lemma \ref{lem:little}, we conclude that
all the derivatives of $z \mapsto \tilde h(z)$ and $z \mapsto \tilde h(ze^{i\theta})$
are real. By looking at the Taylor series, we see that either $\tilde h$ is constant (and thus $h$ is 
constant) or $e^{in\theta}=1$ for some $n>0$.
\endpf

Clearly, a similar generalization of Corollary \ref{cor:c1curves} can be derived.

\subsection{Disjoint circles}

As described in the introduction to this paper,
Perez \cite{perez} proved a more general version of the main result  from \cite{PLB17},
removing the restriction on the singular factor. We   give an alternative proof,
which leads us to more general results in the same area.

Recall that the Nevanlinna class $\GN$ consists of those holomorphic
functions in the unit disc $\DD$ that are expressible as the ratio of two $H^1$ functions.\\

\begin{thm}\label{thm:twoircles}
Let $f,g \in \GN$ satisfy $|f|=|g|$ a.e.\ on $\TT$ and $r\TT$ for
some $0<r<1$. Then $f=\lambda g$ for some $\lambda \in \TT$.
\end{thm}

\beginpf
By dividing out by the common outer factor $u$ given by
\[
u(z) = \exp \left( \frac{1}{2\pi} \int_0^{2\pi} \frac{e^{it}+z}{e^{it}-z} \log |f(e^{it})| \, dt \right)
\]
we may suppose without loss of generality that $f$ and $g$ are inner functions.

Now, if there are any zeros of $f$ or $g$ on $r\TT$ they are at the same points, and finite
in number. By dividing out a finite Blaschke product we may suppose without loss of generality that $f$ and $g$ have no zeros on $r\TT$, so indeed $f/g$ is analytic on an annulus
$\{z: r-\delta < |z| < r+\delta \}$ for some $\delta>0$ and unimodular on $r\TT$.

Indeed, $f/g$ has finitely many poles in $\{z: |z| < r+\delta \}$ so, multiplying by
a function $h$ such that $h(z/r)$ is a finite Blaschke product $B_1(z)$, $fh/g$ is a
function $k$ such that $k(z/r)$ is also a finite Blaschke product $B_2(z)$ since it is analytic
in a neighbourhood of the disc $r\TT$.

Our conclusion is that $f/g$ is a rational function $F(z)=B_2(rz)/B_1(rz)$ in $r\DD$; by analytic continuation
(and removing any isolated singularities)
$f=F g$ throughout the whole disc. 
But if $F$ is non-constant then it is not unimodular on $\TT$ as the zeros and
poles are not placed at inverse points.
\end{proof}

With the aid of the following lemma, which gives a slight extension of
\cite[Lem.~1.3]{kamowitz}, we can prove a stronger result.

\begin{lem}\label{lem:annulus}
Let $F$ be a function meromorphic on an open set containing the unit circle $\TT$
and suppose that $|F(z)|=1$ for infinitely many $z \in \TT$. Then
$|F(z)|=1$ for all $z \in \TT$.
\end{lem}
\beginpf
This result uses a modification of an argument introduced in \cite[Thm. 1.4]{COP24}. Namely,
we   restrict    $F$ to a continuous map from $\TT$ to 
the Riemann sphere $\CC \cup\{\infty\}$, which we continue to denote by $F$. Then $X = F^{-1}(\TT)$ is a closed subset of $\TT$. 
    
If $X$ is an infinite proper subset of $\TT$, we claim that
$X$ contains a point that is a limit of two sequences, one
in $X$ and one in $\TT \setminus X$. This is clear if 
$\TT \setminus X$ consists of a finite number of intervals.
If there are infinitely many intervals, then their endpoints  lie in $X$ and  
accumulate at a point $z_0 \in X$ with the required property.
    \medskip

    This point satisfies $|g(z_0)| = 1$, and there exist two sequences $(u_n) \subset X$ and $(v_n) \subset \TT \backslash X$  which tend to $z_0$, such that $|F(u_n)| = 1$ and $|F(v_n)| \ne 1$
    (by passing to a subsequence we may avoid any poles of $F$). \medskip
    
    By composing $F$ with conformal mappings between $\TT$ and $\RR \cup \{\infty\}$
    we obtain a  function $h$   holomorphic around $0$,  
    and two sequences $(a_n), (b_n) \subset \RR$ such that $a_n \to 0$, $b_n \to 0$, $h(a_n) \in \RR$ and $h(b_n) \not\in \RR$. By Lemma \ref{lem:little}
    we have that
all the derivatives of $h$ at $0$ must be real, so $h(b_n) \in \RR$ (since $b_n \in \RR$). This contradiction shows that  $X$ cannot be an infinite proper subset of $\TT$.
\endpf

\begin{thm}\label{thm:infsub}
Let $f,g \in H^1(\DD)$ satisfy $|f|=|g|$ a.e.\ on $\TT$ and 
on an infinite subset  $X \subset r\TT$  for
some $0<r<1$. 
Then $f=\lambda g$ for some $\lambda \in \TT$.
\end{thm}

\beginpf
Let $f$ and $g$ have inner--outer factorizations $f=u_1 v_1$, $g=u_2v_2$,
where $u_1$ and $u_2$ are inner and $v_1$ are $v_2$ outer with $v_1(0)>0$ and $v_2(0)>0$.
Since
$|f|=|g|$ on $\TT$ we have $v_1=v_2$. We may therefore suppose without
loss of generality that $f$ and $g$ are inner.

Note that the function $F:=f/g$ is meromorphic in the disc
and  $|F(z)|=1$ for an infinite subset of $\TT$.

By Lemma \ref{lem:annulus},  $|F|=1$ on $\TT$ and so
 $|f|=|g|$ on $r\TT$.
The result now follows from Theorem \ref{thm:twoircles}.

\endpf
\begin{cor}	\label{cor:offcentre}
Let $f,g \in H^1(\DD)$ (or more generally in the Nevanlinna class) satisfy $|f|=|g|$ a.e.\ on $\TT$ and 
on an infinite subset  $X \subset \Poa (r\TT)$  for
some $0<r<1$, where $\omega\in \TT$, $\alpha\in \DD$, and $\Poa(z):=\omega\frac{\alpha -z}{1-\overline{\alpha}z}$.
Then $f=\lambda g$ for some $\lambda \in \TT$.	
\end{cor}
\begin{proof}
We apply Theorem~\ref{thm:infsub} to $f\circ \Poa$ and $g\circ \Poa$, which implies  that $f\circ  \Poa =g\circ  \Poa$, and thus $f=g$ since $ \Poa$ is an automorphism of the open unit disc. 
\end{proof}

Note that every circle contained in $\DD$ can be written
as $ \Poa(r\TT)$ for 
a suitable choice of $\omega \in \TT$, $\alpha \in \DD$, and $0<r<1$.

\begin{rem}
It is clearly not enough to suppose that
 $|f|=|g|$ a.e.\ on $\TT$ and 
on some finite subset  $X \subset r\TT$. For example, we can let $f$ and $g$ be two different inner functions
 constructed in the following way:
  \[f=\Psi_\alpha \circ(Bu)\mbox{ and }g=\Psi_\alpha \circ (Bv),\]
where $B$ is the finite Blaschke product associated with the finite set $X$, $u$ and $v$ are inner functions with $u\neq cv$ for any $c\in\TT$ and $\Psi_\alpha (z):=\frac{\alpha -z}{1-\overline{\alpha}z}$ for an arbitrary  
$\alpha\in\DD$.  Since $f$ and $g$ are inner, their radial limits are of modulus one almost everywhere on $\TT$, $f(z)=\alpha =g(z)$ for all $z\in X$.  The choice of $u$ and $v$ implies that there is no $c\in\TT$ such that $f=cg$.

\end{rem}

If the inner parts of $f$ and $g$ are finite Blaschke products $B_1$ and $B_2$, then, depending on their degrees, $|f|=|g|$ a.e.\ on $\TT$ and at only finitely many
points on $r\TT$ is a sufficient condition for Theorem \ref{thm:infsub} to hold.

\begin{thm}\label{thm:BPpolyeq}
Let $B_1$ and $B_2$ be Blaschke products of degrees $M$ and $N$, respectively
and take $0<r<1$.
If $|B_1(z)|=|B_2(z)|$ for more than $2N+2M-1$ distinct points on $r\TT$, then
$B_2$ is a constant multiple of $B_1$.
\end{thm}

\beginpf
Note that on $r\TT$ we have $\bar z=r^2/z$;
so for an elementary Blaschke factor with $a \in \DD$
we have
\[
\frac{z-a}{1-\bar a z}\frac{\bar z-\bar a}{1-a\bar z}=
\frac{(z-a)(r^2-\bar a z)}{(1-\bar a z)(z-r^2 a)}
\] and thus
the equation
\[
B_1(z)\overline{B_1(z)}=B_2(z) \overline {B_2(z)}
\]
can be rewritten as
$
R_1(z)=R_2(z)$, where $R_1$ and $R_2$ are rational functions of degrees $2M$ and $2N$
respectively. 
If $\alpha_1,\cdots , \alpha_M$ are the zeroes of $B_1$ and $\beta_1,\cdots, \beta_N$ are the zeroes of $B_2$, we get 
\[ R_1(z)=\left(\frac{z-\alpha_1}{1-\overline{\alpha_1} z}\right) \cdots \left( \frac{z-\alpha_M}{1-\overline{\alpha_M} z}\right) \left( \frac{r^2-\overline{\alpha_1}z}{z-\alpha_1 r^2}\right) \cdots \left( \frac{r^2-\overline{\alpha_M}z}{z-\alpha_M r^2}\right) \] 
and 
\[ R_2(z)=\left(\frac{z-\beta_1}{1-\overline{\beta_1} z}\right) \cdots \left( \frac{z-\beta_N}{1-\overline{\beta_N} z}\right) \left( \frac{r^2-\overline{\beta_1}z}{z-\beta_1 r^2}\right) \cdots \left( \frac{r^2-\overline{\beta_N}z}{z-\beta_N r^2}\right).\] 
Denote by $P_1(z)$ (resp. $P_2(z)$) the numerator of $R_1(z)$ (resp. $R_2(z)$)  and $Q_1(z)$ (resp. $Q_2(z)$) the denominator of $R_1(z)$ (resp. $R_2(z)$).  
This reduces to a polynomial equation of degree at most $2N+2M-1$, noting that the coefficient of $z^{2N+2M}$ of the polynomial $P_1Q_2$  is $(-1)^{N+M} \prod_{i=1}^{M} \overline{\alpha_i }\prod_{j=1}^N \overline{\beta_j}$, which coincide with the coefficient of  $z^{2N+2M}$ of the polynomial $P_2Q_1$.
Therefore if  $P_1Q_2-P_2 Q_1$it has more than $2N+2M-1$ roots  then it is identically zero.
That is, $|B_1|=|B_2|$ on the whole of $r\TT$ and thus the result follows
from the previous discussions.
\endpf

In general we have the following parametrization of  functions of equal modulus
on $r\TT$.

\begin{thm}\label{thm:paramrt}
Suppose that $f,g \in \Hol(\DD)$ and $|f|=|g|$ on $r\TT$ for some
$0<r<1$. Then there exist finite Blaschke products $B_1,B_2$ such
that
\[
B_1(z/r) f(z)=B_2(z/r) g(z).
\]
\end{thm}
\beginpf
We have that $f/g$ is meromorphic on a neighbourhood of $r\overline\DD$ and
$|f/g|=1$ on $r\TT$. 
By choosing a suitable  finite Blaschke product $B_1$
and writing $\tilde B_1(z)=B_1(z/r)$
we have that $\tilde B_1 f/g$
is still unimodular on $r\TT$ and holomorphic on a neighbourhood of $r\overline \DD$.
Thus it has the form $\tilde B_2$, where $\tilde B_2(z)=B_2(z/r)$ for some 
finite Blaschke product $B_2$.
\endpf

\subsection{Disjoint lines}

In \cite{PLB17} it is claimed that a similar result
to Theorem  \ref{thm:twoircles} can be proved for the corresponding space $H^1_R(\CC^+)$
of analytic functions in the upper half-plane $\CC^+$, 
considering the values on $\RR$ and $i+\RR$,
although no proof is given. The result cannot be deduced directly from the disc
result as this strip
$\MS=\{s \in \CC: 0 < \re s < 1\}$
 is not conformally equivalent to the annulus.

We note also that there are non-constant meromorphic functions $F$
defined on $\CC^+$ such that $|F|=1$ on both $\RR$ and $i+\RR$.
One example is 
\[
F(s)= \frac{i-\exp(\pi s)}{i+\exp(\pi s )},
\]
which provides a conformal map of the strip $\MS$
onto the unit disc. However the function does not extend to a quotient
of $H^p$ functions, since its zeros 
$\{i(\frac12+2n): n \ge 0\}$
do  not form a Blaschke sequence.

By means of the Weierstrass factorization theorem, $F$ can be modified
to give distinct analytic  functions $G$ and $H$ on $\CC^+$  such that
$|G|=|H|$ on $\RR$ and $1+i\RR$.\\

In fact, the result we require does hold for the half-plane. It will be convenient to work
with the right half-plane $\CC_+$.

\begin{thm}\label{thm:halfp}
Suppose that $f,g \in H^1(\CC_+)$ and $|f|=|g|$ a.e.\ on $i\RR$, while
 $|f|=|g|$ on $1+i\RR$. Then $f=cg$ for some
 unimodular constant $c$.
 \end{thm}
 
 \beginpf
 
 We may assume without loss of generality that $f,g$ are inner.
 
 Now, the functions $F$ and $G$ defined by $F(s)=f(1+s)$ and $G(s)=g(1+s)$
 satisfy
 $|F|=|G|$ on $ i\RR$; and these functions
 are also in $H^1(\CC_+)$, and holomorphic in a neighbourhood of $i\RR$.
 Let us consider the inner--outer factorization $F=uh$, say.
 
  Clearly, the inner factor $u$ has no zeros accumulating at a point on $i\RR$.
  Moreover, the outer factor $h$ is continuous on the closed half-plane $\CC_+$
  (see, for example, the arguments in \cite[Sec.\ 4.3.8]{nik}), which means
  that the inner factor $u$ is also continuous except possibly at the discrete
  set consisting of the
  zeros of $uh$ on $i\RR$. 
On factoring out a zero at $iy_0$, by
considering $uh/b$, where $b(s)= \dfrac{s-iy_0}{1+s}$ (an outer function)
we see that $u$ is also continuous at the zeros of $uh$. Thus $u$ has no singular part
except possibly an exponential factor $e^{-\alpha s}$ for some $\alpha \ge 0$.
We conclude that $u(s)=B_1(s)e^{-\alpha s}$ 
where $B_1$ is a  Blaschke product   whose zeros do not accumulate at any point 
in $\ol{\CC_+}$. The function
$B_1$, which is a product of factors of the form $\dfrac{s-a_n}{s+ \ol{a_n}}$,
therefore has a meromorphic extension to $\CC$.
 
So suppose that $B_1$ has a zero at $w$. Then 
the Blaschke factor of $f$, say, $b_1$,  has a zero at $w+1$, and hence
a pole at $-\ol w-1$. Thus $B_1$ has a pole at $ -\ol w-2$ and hence a zero at $w+2$.
 
Repeating this argument, we find that $B_1$ has zeros at $w_n:=w+2n$, $n \in \NN$. But these do not satisfy
the Blaschke condition $\sum_{n=0}^\infty \frac{\re w_n}{1+|w_n|^2} < \infty$,
and that is a contradiction.
 
 So $B_1$ is a unimodular constant and
 we may apply a similar argument to the inner function $v$.
We conclude that  
 $f(s)/g(s)= c\exp(\lambda(s-1))$, 
 for a unimodular constant $c$,
 which, since
 $f$ and $g$ are inner,  means that 
 $|\exp(\lambda(iy-1))|=1$ a.e., and $\lambda=0$.
 \endpf
 
 The translation of the previous theorem into the open unit disc via a standard conformal map between the unit disc and the right half-plane is the following. 
 
\begin{cor}\label{cor:fin}
	Suppose that $f,g \in H^1(\DD)$ and $|f|=|g|$ a.e.\ on $\TT$, while
	$|f|=|g|$ on $\{z\in\CC: |z-r|=1-r\}$ for some $r\in (0,1)$. Then $f=cg$ for some
	unimodular constant $c$.
\end{cor}	

\begin{rem}
 Corollary~\ref{cor:fin} also holds  when $f$ and $g$ are in thee Nevalinna class $\mathcal N$ and when 
 	$|f|=|g|$ on $\TT$ and on a sequence included in $\{z\in\CC: |z-r|=1-r\}$.
 \end{rem}
 
 \subsection{Two circles contained in the unit disc}
 
 Suppose now that $f$ and $g$ are holomorphic on $\DD$ and $|f|=|g|$ on two distinct circles
 $C_1,C_2$ contained in $\DD$, which bound open discs $D_1$ and $D_2$ respectively.
 There are five cases to consider (we can of course swap the roles of $C_1$ and $C_2$
 if we wish):
 \begin{enumerate}
 \item $C_1$ and $C_2$ are ``internally'' disjoint with $C_2 \subset D_1$;
\item $C_1$ and $C_2$ are ``externally'' disjoint with $C_2 \subset \DD \setminus \ol{D_1}$;
\item  $C_1$ and $C_2$ are ``internally'' tangential with $C_2 \subset \ol{D_1}$;
\item  $C_1$ and $C_2$ are ``externally'' tangential with $C_2 \subset \DD \setminus D_1$;
\item $C_1$ and $C_2$ intersect in two points, at which they make an angle $\theta \in (0,\pi/2]$.
\end{enumerate}
	\begin{center}
	\includegraphics[width=12cm]{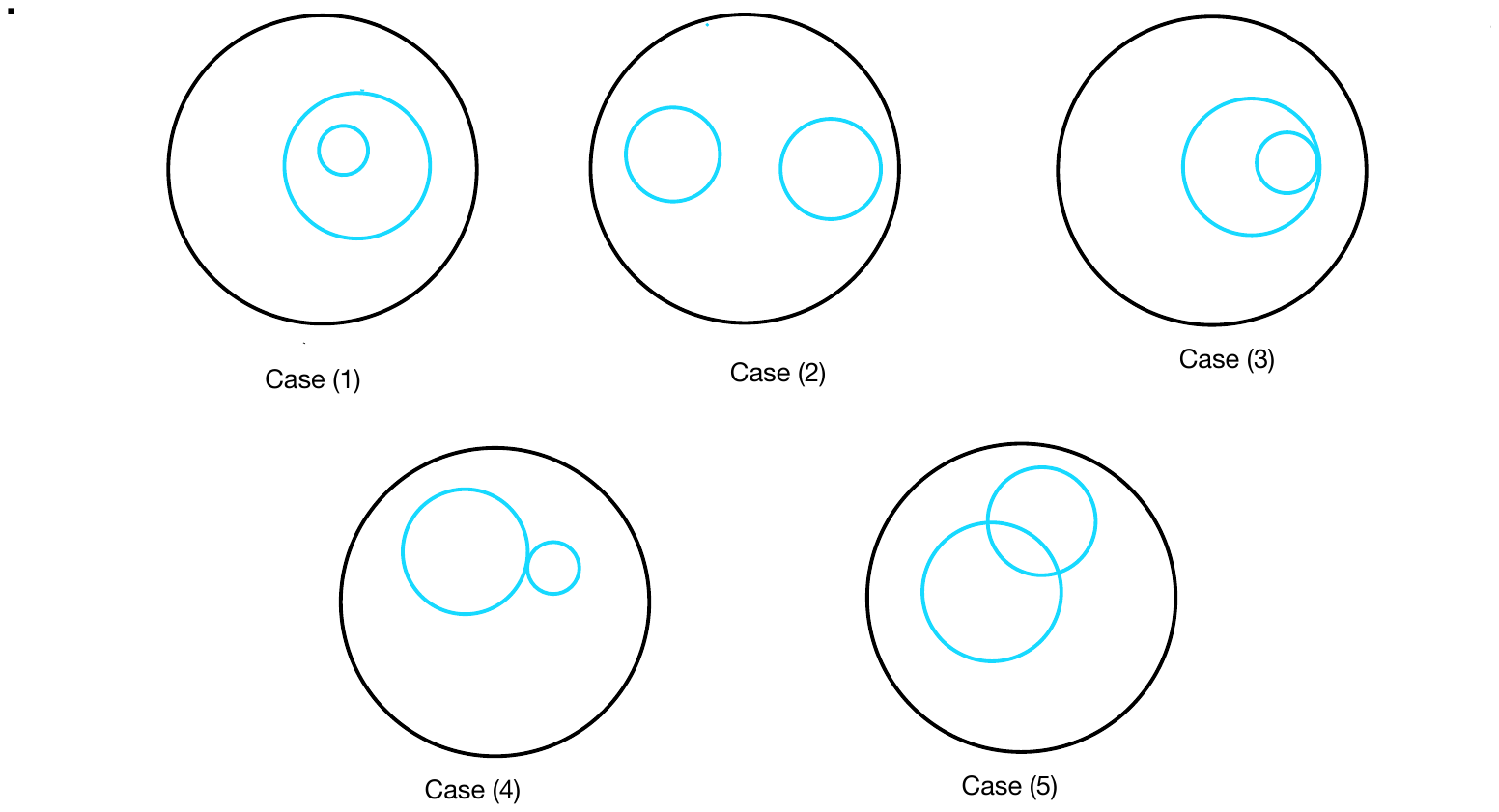}
	
\vspace{-2cm}	The five cases for two circles contained in $\DD$. 
\end{center}
\begin{thm}\label{thm:fivecases}
Suppose that $C_1$ and $C_2$ satisfy one of the five conditions listed above
and that $|f|=|g|$ on $C_1 \cup C_2$. Then in cases (1)--(4)
$f=cg$ for some $c$ with $|c|=1$. The same holds in case (5) if $\theta$ is
an irrational multiple of $\pi$, but need not hold if $\theta$ is a rational
multiple of~$\pi$.
\end{thm}
\beginpf
For parts (1)--(4) we may suppose without loss of generality,
by composing with an automorphism, that $C_1$ is the circle $r\TT$ centred at $0$
with radius $r$ for some $0<r<1$. For (1) we   now consider $f_r$ and $g_r$
defined by $f_r(z)=f(z/r)$ and $g_r(z)=g(z/r)$.
Then
(1) follows easily from Corollary \ref{cor:offcentre} and the comment following it.\\

For (2) with $C_1=r\TT$ we
note from Theorem \ref{thm:paramrt} that $f(z)/g(z)$ is a rational function of the form $B_2(z/r)/B_1(z/r)$ and hence meromorphic on $\CC \cup \{\infty\}$.
With the change of variable $w=z/r$ we have a meromorphic function $B_2(w)/B_1(w)$
that is unimodular on a smaller circle contained in $\DD$. By composition with an automorphism
we may supposed that the smaller circle is centred at $0$. But now the fact that $f/g$ is
a constant follows
an argument similar to that used in the proof of Theorem \ref{thm:twoircles} (specifically, that
poles and zeros occur in pairs of inverse points with respect to
one circle, which means that they cannot occur at inverse points with respect to
the other circle).\\

(3) is easily derived from Corollary \ref{cor:fin}.\\

For (4) we may again suppose that $C_1=r\TT$ and that $f(z)/g(z)=B_2(z/r)/B_1(z/r)$.
Now we take $w=z/r$ again, reducing to the case when a function
of the form $B_2(w)/B_1(w)$ is of modulus 1 on a circle internally tangential
to $\TT$. By means of a conformal mapping we
can transform this to the half-plane, producing a rational function
that has absolute value 1 on the lines $i\RR$ and $1+i\RR$.
Finally, an argument based on the fact that zeros and poles must occur at pairs of inverse points
$a$ and  $-\ol a$ with respect to $i\RR$ as well
as inverse points $ a$ and $-\ol a+2$ with respect to $1+i\RR$
(cf. the proof of Theorem \ref{thm:halfp})
 leads
to a contradiction unless $f/g$ is constant.\\

For (5), suppose that the two intersection points of $C_1$ and $C_2$ are $a$ and $b$.
By changing the variable to $w=1/(z-a)$ we transform the
circles into straight lines, which meet at $1/(b-a)$, still at an angle $\theta$.
The result for irrational $\theta/\pi$ now follows from Theorem \ref{thm:first}, applied to a small
disc centred at $1/(b-a)$.

We may similarly construct counterexamples for rational $\theta/\pi$.
Let $C_1$ and $C_2$ be circles of radius $1/3$ centered at $\pm 1/(3\sqrt{2})$.
Then they meet at $\pm a$, where $a=i/(3\sqrt 2)$. The angle between them
is a right angle.
The transformation $w= \dfrac{z+a}{z-a}$ sends the circles to two perpendicular lines,
meeting at $0$, from which  examples  similar to those in Example \ref{ex:22}
can be constructed easily.
\endpf

\begin{ex}
There are limitations to the ``inverse points'' argument above
if two circles are not concentric. For example, let $C_1=\{z \in \CC: |z-3/5|=1/5\}$
and $C_2=\{z \in \CC: |z+3/5|=1/5\}$. Then  the points $z_\pm:=\pm \sqrt{8}/5$ are inverse
with respect to both circles, and one can construct rational functions of the form
$c \dfrac{z-z_+}{z-z_-}$ that have constant absolute values on both
circles. However, they do not have the same absolute values on $C_1$ and $C_2$,
so do not provide a counterexample to the above result.
\end{ex}

Finally, we note that Theorem \ref{thm:fivecases} has sequential counterparts, which may be proved using the
methods of Subsection \ref{sec:2.1}. We omit the details.


\begin{thebibliography}{10}

\bibitem{COP24}
I. Chalendar, L. Oger and J.R. Partington, Linear isometries of $\Hol(\DD)$. Preprint, 2024.
{\tt https://arxiv.org/abs/2402.14671}

\bibitem{duren}
P.L. Duren,  
Theory of $H^p$ spaces.
Pure and Applied Mathematics, Vol. 38 Academic Press, New York--London 1970.

\bibitem{JKP20}
P. Jaming, K. Kellay, R. Perez   III, Phase retrieval for wide band signals. 
{\em J. Fourier Anal. Appl.} 26 (2020), no. 4, Paper No. 54, 21 pp. 

\bibitem{kamowitz}
H. Kamowitz,  
The spectra of composition operators on $H^p$.
{\em J. Functional Analysis\/} 18 (1975), 132--150.

\bibitem{liehr}
L. Liehr,
Arithmetic progressions and holomorphic phase retrieval.
Preprint, 2023. {\tt https://arxiv.org/abs/2308.05722}


\bibitem{M04}
J.N. McDonald,  
Phase retrieval and magnitude retrieval of entire functions.
{\em J. Fourier Anal. Appl.} 10 (2004), no.3, 259--267.

\bibitem{nik}
N.K.~Nikolski,
{\em Operators, functions, and systems: an easy reading. {Volume} {I}: {Hardy}, {Hankel}, and {Toeplitz}. {Transl}. from the {French} by {Andreas} {Hartmann}}.
 English. Vol. 92. Math. Surv. Monogr.
Providence, RI: American Mathematical Society (AMS), 2002.

\bibitem{perez}
R. Perez  III, A note on the phase retrieval of holomorphic functions. 
{\em Canad. Math. Bull.} 64 (2021), no. 4, 779--786.


\bibitem{PLB17}
V. Pohl, N. Li and H. Boche, Phase retrieval in spaces of analytic functions
on the unit disk,  {\em IEEE Proc.   International Conference on Sampling Theory and Applications (SampTA)}, 2017.

\bibitem{smith}
M. Smith, 
The spectral theory of Toeplitz operators applied to approximation problems in Hilbert spaces. 
{\em Constr. Approx.} 22 (2005), no. 1, 47--65.


\end{thebibliography}
\end{document}